\documentclass{amsart}

\usepackage{amsmath, amssymb, color}

\newtheorem{theorem}{Theorem}
\newtheorem{prop}[theorem]{Proposition}
\newtheorem{lemma}[theorem]{Lemma}

\theoremstyle{definition}
\newtheorem{definition}[theorem]{Definition}

\theoremstyle{remark}

\numberwithin{equation}{section}

     \title[Alternative summations for $G_2$ and $\wp$]{Alternative summation orders for the Eisenstein series~$G_2$ and Weierstrass $\boldsymbol{\wp}$-function}

     \author{Dan Romik}
     \address{Department of Mathematics\\
  University of California, Davis, One Shields Avenue, Davis, CA 95616, USA}

     \email[]{romik@math.ucdavis.edu}

     \author{Robert Scherer}
     \email[]{rscherer@math.ucdavis.edu}
     \thanks{This material is based upon work supported by the National Science Foundation under Grant No. DMS-1800725.}

\begin{document}
     \begin{abstract}
We consider alternative orders of summation for the conditionally convergent series defining the weight-2 Eisenstein series $G_2$ and the Weierstrass $\wp$-function. The resulting sums differ from the standard ones by a residual term that can be thought of as a function of the shapes with respect to which we sum. We compute this residual function explicitly and give some examples. The results generalize the well-known quasimodularity relationship between $G_2$ and its series summed in the reverse order.
\end{abstract}
     \maketitle
{

\section{Introduction}

\label{sec:intro}

Students of analysis learn that series that are conditionally convergent can be rearranged in a way that changes their value. In the case of series of real numbers, this is the well-known Riemann series theorem, that states that a rearrangement can change the value of a conditionally convergent series to an arbitrary real number \cite[Ch.~3]{rudin}; this is generalized for conditionally convergent series of complex numbers, or more generally vectors of arbitrary dimension, by a (less well-known) result known as the L\'evy-Steinitz theorem~\cite{rosenthal}. 

The examples typically given in an undergraduate calculus or analysis course to illustrate this rearrangement phenomenon have a somewhat uninspiring and contrived feel to them, which may lead some students encountering these ideas for the first time to wonder if this phenomenon is something they really need to worry about: does it belong to the realm of ``pathological'' examples, or does it come up in actual examples that appear ``in nature''? In fact, in complex analysis one encounters a particularly natural and explicit example of the rearrangement of a conditionally convergent series, involving the \textbf{weight-$2$ Eisenstein series} $G_2$ (also referred to in \cite{Stein} as the \textbf{forbidden Eisenstein series}). Recall that $G_2$ is the holomorphic function on the upper half-plane $\mathbb{H}=\{\tau\in\mathbb{C}:\text{Im}(\tau)>0\}$ defined by
\begin{equation} \label{eq:def-g2}
G_2(\tau)=\sum_{m} \left[\sum_{n} \frac{1}{(n+m\tau)^2}\right],
\end{equation}
summed in the indicated order over all $m,n\in\mathbb{Z}$, except for $(m, n)= (0,0)$ where the summand is undefined (see \cite[Ch.~1]{Apostol}). 
The series \eqref{eq:def-g2} does not converge absolutely; moreover, Gotthold Eisenstein proved in 1847 \cite[pp.~416--418]{eisenstein} (see also \cite[p.~127]{roy}) that switching the order of summation in the double sum yields
\begin{equation}
\label{eq:g2-order-summation}
\sum_{n}\left[\sum_{m}  \frac{1}{(n+m\tau)^2}\right]=\sum_{m}\left[\sum_{n} \frac{1}{(n+m\tau)^2}\right]-\frac{2\pi i}{\tau}
\end{equation}
(with both series excluding $(m,n)=(0,0)$ as above).
That is, the change in the order of summation results in a predictable change in the value of the series, given by the ``residual term'' $-2\pi i/\tau$. (Interestingly, Eisenstein's result predates by several years Riemann's more general rearrangement theorem, which he discovered around 1852--1853 after hearing from Dirichlet about another specific instance of the phenomenon; see \cite[p.~124]{grattan-guinness}.)

The identity~\eqref{eq:g2-order-summation} arises naturally when one considers the behavior of the Eisenstein series under the modular transformation $\tau\mapsto -1/\tau$. Specifically, a key property of $G_2$ is that it satisfies
the so-called quasimodularity identity (derived in the exercises of \cite[Ch.~3]{Apostol})
\begin{equation}
\label{eq:g2-order-summation-equivalent}
\tau^{-2}G_2(-1/\tau)=G_2(\tau)-\frac{2\pi i}{\tau}.
\end{equation}
In fact, evaluating $G_2$, as defined above, at $-1/\tau$ and manipulating the double sum shows that 
\[
\tau^{-2}G_2(-1/\tau)=\sum_{n}\left[\sum_{m}  \frac{1}{(n+m\tau)^2}\right],\]
so \eqref{eq:g2-order-summation-equivalent} is equivalent to \eqref{eq:g2-order-summation}.
Another identity, similar to \eqref{eq:g2-order-summation}, occurs when considering another well-known function from complex analysis, the \textbf{Weierstrass $\wp$-function}.
Specifically (see Proposition~\ref{prop:wp-twoorders} below)
\[
\sum_{n}\left[\sum_{m}\frac{1}{(z+n+m\tau)^2}\right] =\sum_{m}\left[\sum_{n}\frac{1}{(z+n+m\tau)^2}\right]-\frac{2\pi i}{\tau},
\]
for $\tau \in \mathbb{H}$, $z\not\in\mathbb{Z}\tau+\mathbb{Z}$, where in this case there is no need to exclude $(m,n)=(0,0)$ from the summations.

Our goal in this paper is to show that the notion of a residual term can be generalized to a much larger class of orders of summation for the series defining $G_2$ and the Weierstrass $\wp$-function. Namely, we will assign to certain compact shapes in $\mathbb{R}^2$ the residual term that arises from partially summing the relevant infinite series over the integer lattice points $(m,n)$ inside a scaled-up copy of the shape and taking the limit of these sums as the scaling factor goes to infinity.
One of our main results (Theorem~\ref{main-thm} below) is an explicit formula for this residual term.

To this end, denote by $\mathcal{K}$ the class of compact sets $K \subset \mathbb{R}^2$ that are convex, have nonempty interior and are symmetric about the $x$ and $y$ axes. (For simplicity we restrict the discussion to this class of shapes, although it is possible to consider things at a greater level of generality; see the final comment in Section~\ref{sec:concluding}.)

We now come to the key definitions.
\newcommand{\shapesum}{\sum\limits}

\begin{definition}
For each $K\in\mathcal{K}$ we define $h_K$ to be the real-valued function whose graph is the upper boundary of the shape $K$. The function $h_K$ is necessarily compactly supported on an interval of the form $[-A,A]$, is an even function, and its reflection $-h_K$ is the lower boundary of $K$.
\end{definition}

\begin{definition}
For a shape $K \in \mathcal{K}$ and an array $(a_{m,n})_{m,n\in \mathbb{Z}}$ of complex numbers,
we define
\begin{equation} \label{eq:shape-summation}
\shapesum_K a_{m,n} := \lim_{\lambda\to \infty} \sum_{(m,n) \in (\lambda K )\cap \mathbb{Z}^2} a_{m,n},
\end{equation}
provided the limit exists. We refer to this sum as the \textbf{$K$-summation}, or shape summation with respect to the shape $K$, of the array $(a_{m,n})$.
\end{definition}
In the next definition we apply the concept of shape summation in a way that generalizes the definition of the Eisenstein series $G_2$. (The case of the Weierstrass $\wp$-function will be discussed in Section~\ref{sec:weierstrass}.)\smallskip\\

\begin{definition}
If $K\in \mathcal{K}$,
we denote by $G_2(K,\tau)$ the \textbf{$K$-summation of the weight-2 Eisenstein series}, defined as
\begin{equation}
\label{eq:eisenstein-shape-summation}
G_2(K,\tau):=\sum_{K}\frac{1}{(m\tau+n)^2},
\end{equation}
provided the limit defining the summation exists, and with the convention that $a_{0,0}=0$ in \eqref{eq:shape-summation}, to make allowance for the fact that the summand $\frac{1}{(m\tau+n)^2}$ is not defined for $m=n=0$.

\end{definition}

\begin{definition}
If $K\in\mathcal{K}$ and $G_2(K,\tau)$ is defined, 
we denote by $E(K,\tau)$ the \textbf{residual function associated to $K$}, which is defined as 
\[
E(K,\tau):=G_2(K,\tau)-G_2(\tau).
\]
\end{definition}

With these definitions in place we have the following result, which gives an explicit formula for the residual function.

\begin{theorem}
\label{main-thm}
For all $\tau\in\mathbb{H}$ and all $K\in\mathcal{K}$, the limit defining $G_2(K,\tau)$ exists. The residual function $E(K,\tau)$ is given by
\begin{equation} \label{eq:residual}
E(K,\tau)=4\int_0^A\frac{h_K(x)}{\tau^2x^2-h_K^2(x)}\, dx 
\end{equation}
(where as before, $A$ denotes a number for which $h_K$ is supported on $[-A,A]$).
\end{theorem}

\noindent\textbf{A motivating example: rectangles.} To motivate the result, consider first the simplest example, namely when $K$ is the rectangle $[-c,c]\times[-1,1]$ with aspect ratio $c$, for some $c>0$. In this case, $h_K$ is the indicator function $h_K(x)=\chi_{[-c,c]}(x)$. Evaluating the integral in \eqref{eq:residual} gives that
$$ E(K,\tau)=G_2(K,\tau)-G_2(\tau)=-\frac{4}{\tau}\tanh^{-1}(c \tau).$$
In terms of the principal branch of the logarithm, the latter expression is
\[
-\frac{2}{\tau}[\log(1+c\tau)-\log(1-c\tau)].
\]
This already seems interesting, since in particular note that we can interpret the limiting case $c\to0$ of the shape summation \eqref{eq:eisenstein-shape-summation} to represent a summation with respect to an ``infinitely tall and narrow'' rectangle, that is, first summing over $n$ and then over $m$ as in the original definition \eqref{eq:def-g2} of $G_2(\tau)$. The residual function in that case should be $0$, and indeed we have that
$\lim_{c\to 0}  -\frac{4}{\tau}\tanh^{-1}\left(c \tau \right)=0$.
At the other extreme, we can interpret the case $c\to \infty$ to represent summing with respect to an ``infinitely long and thin'' rectangle, that is, first summing over $m$ and then $n$. 
In this case we have that
$\lim_{c\to\infty} -\frac{4}{\tau}\tanh^{-1}\left(c \tau \right)=-\frac{2\pi i}{\tau}$, and indeed this is consistent with the relation \eqref{eq:g2-order-summation}, which can now be understood as giving the residual function of such long and thin rectangles. Thus we see that summing with respect to rectangles provides a conceptual generalization of \eqref{eq:g2-order-summation}.

\section{Proof of Theorem \ref{main-thm}}

A key ingredient in the proof is the following Lemma \ref{key-lemma}, which shows that the right-hand side of \eqref{eq:residual} is the $K$-summation of a different series than \eqref{eq:eisenstein-shape-summation}, which has the advantage of being a telescoping series in the summation index~$n$. In proving Theorem \ref{main-thm}, we will use this series to write $G_2(\tau)$ as a series that converges absolutely, which will allow us to compare it with $G_2(K,\tau)$ and show that the series in Lemma \ref{key-lemma} coincides with the residual function $E(K,\tau)$.
\begin{lemma}
\label{key-lemma}
Let $K\in \mathcal{K}$ be a shape with corresponding function $h_K,$ supported on $[-A,A].$
Then
\[\sum_{K} \left(\frac{1}{m\tau+n}-\frac{1}{m\tau+n+1}\right)=4\int_0^A \frac{h_K(x)}{\tau^2x^2-h_K^2(x)}\, dx,\]
where we exclude summands corresponding to $m=0$, that is, we set $a_{0,n}=0$ for all $n$ in \eqref{eq:shape-summation}.
\end{lemma}

\begin{proof} For brevity of notation, we set $h=h_K$ for the rest of the proof. To evaluate the sum in the lemma,
we rewrite it as 
\begin{align}
\label{eq:rewrite-1}
&\sum_{K} \left(\frac{1}{m\tau+n}-\frac{1}{m\tau+n+1}\right)\\
&=\lim_{\lambda\to\infty}\sum_{\substack{-\lambda A\leq m\leq \lambda A\\ m\neq 0}}\,\,\sum_{-\lambda h(m/\lambda)\leq n\leq \lambda h(m/\lambda)}\left(\frac{1}{m\tau+n}-\frac{1}{m\tau+n+1}\right). \nonumber
\end{align}
\noindent After telescoping the inner summation, this becomes
\begin{align*}
&\lim_{\lambda\to\infty}\sum_{\substack{-\lambda A\leq m\leq \lambda A\\ m\neq 0}}
\Big(\frac{1}{m\tau - \lfloor\lambda h(\frac{m}{\lambda})\rfloor} - \frac{1}{m\tau +\lfloor\lambda h(\frac{m}{\lambda})\rfloor}
+ \frac{1}{m\tau +\lfloor \lambda h(\frac{m}{\lambda})\rfloor} 
\\ & \hspace{220pt}
- \frac{1}{m\tau + \lfloor\lambda  h(\frac{m}{\lambda})\rfloor+1}\Big)\\
&=\lim_{\lambda\to\infty}\sum_{\substack{-\lambda A\leq m\leq \lambda A\\ m\neq 0}}
\left[\left(\frac{ 2 \lfloor\lambda  h(\frac{m}{\lambda})\rfloor }{m^2\tau^2-\lfloor\lambda h(\frac{m}{\lambda})\rfloor ^2}\right)
+\frac{1}{(m\tau+\lfloor\lambda h(\frac{m}{\lambda})\rfloor)
(m\tau+\lfloor\lambda h(\frac{m}{\lambda})\rfloor+1)}\right]\\
&=2\lim_{\lambda\to\infty}\sum_{1\leq m\leq \lambda A}\left(\frac{ 2 \lfloor\lambda  h(\frac{m}{\lambda})\rfloor }{m^2\tau^2-\lfloor\lambda h(\frac{m}{\lambda})\rfloor ^2}\right)\\
&=4\lim_{\lambda\to\infty}\frac{1}{\lambda}\sum_{1\leq m\leq \lambda A}
\left(\frac{ \lambda^{-1}\lfloor\lambda  h(\frac{m}{\lambda})\rfloor }{\lambda^{-2}m^2\tau^2-\lambda^{-2}\lfloor\lambda h(\frac{m}{\lambda})\rfloor ^2}\right).
\end{align*}

Since, $\lambda^{-1}\lfloor\lambda  h(\frac{m}{\lambda})\rfloor\sim h(\frac{m}{\lambda})$ and $\lambda^{-2}\lfloor\lambda  h(\frac{m}{\lambda})\rfloor ^2\sim h(\frac{m}{\lambda})^2$ as $\lambda \to \infty,$ the above limit is the same as the limit of a Riemann sum,
\[
4\lim_{\lambda\to\infty}\frac{1}{\lambda}\sum_{1\leq m\leq \lambda A }
\left(\frac{ h(\frac{m}{\lambda}) }{\lambda^{-2}m^2\tau^2- h^2(\frac{m}{\lambda})}\right),
\]
which is the integral in the lemma.
\end{proof}

We now observe that
\[
\sum_{m\neq 0}\left[\sum_{n\in\mathbb{Z}}\left(\frac{1}{m\tau+n} -\frac{1}{m\tau +n+1}\right)\right]=0.
\]
To see this, note that for any $m\neq 0$, the inner sum converges absolutely  (since the summands are $O(1/n^2)$ as $n\to \infty$), and is equal to   
\[
\lim_{N\to\infty} \sum_{n=-N}^{N-1} \left(\frac{1}{m\tau+n} -\frac{1}{m\tau +n+1}\right)=
\lim_{N\to\infty} 
\left(\frac{1}{m\tau-N} -\frac{1}{m\tau+N}\right)=0.
\]

Thus, we can write $G_2(\tau)$ in a different form, namely
\begin{align*} 
G_2(\tau)- \sum_{n\neq0}\frac{1}{n^2}&=\sum_{m\neq 0}\left[\sum_{n\in\mathbb{Z}} \frac{1}{(n+m\tau)^2}\right]\\
&=\sum_{m\neq 0} \left[\sum_{n\in\mathbb{Z}} \frac{1}{(n+m\tau)^2}\right]-
\sum_{m\neq 0}\left[\sum_{n\in\mathbb{Z}}\left(\frac{1}{m\tau+n} -\frac{1}{m\tau +n+1}\right)\right]\\
&=\sum_{m\neq 0}\left[\sum_{n\in\mathbb{Z}}\frac{1}{(m\tau+n)^2(m\tau+n+1)}\right].
\end{align*}
The latter series has the advantage of converging absolutely by comparison with $\sum_{m}\sum_{n}\frac{1}{(m\tau+n)^3}$ (see Lemma 1.1 in [2]).

Next we observe, still setting $h=h_K$, that Definition 3 can be expressed as 

\begin{align}\label{eq:rewrite-2}
G(K,\tau)=&\sum_{n\neq 0}\frac{1}{n^2}\\
&+\lim_{\lambda\to\infty}\sum_{\substack{-\lambda A\leq m\leq \lambda A\\ m\neq 0}}\,\,\,\,\sum_{-\lambda h(m/\lambda)\leq n\leq \lambda h(m/\lambda)} \frac{1}{(m\tau+n)^2}.\notag\\
\notag
\end{align}

By combining \eqref{eq:rewrite-1} and \eqref{eq:rewrite-2} , we obtain
\begin{align*}
&G_2(K,\tau)-\sum_{n\neq0}\frac{1}{n^2}-\sum_{K} \left(\frac{1}{m\tau+n}-\frac{1}{m\tau+n+1}\right)\\
&=\lim_{\lambda\to\infty}\sum_{\substack{-\lambda A\leq m\leq \lambda A\\ m\neq 0}}\,\,\,\sum_{-\lambda h(m/\lambda)\leq n\leq \lambda h(m/\lambda)} \left(\frac{1}{(m\tau+n)^2}- \frac{1}{m\tau+n}+\frac{1}{m\tau+n+1}\right)\\
&=\lim_{\lambda\to\infty} \left[\sum_{\substack{-\lambda A\leq m\leq \lambda A\\ m\neq 0}}\,\,\,\sum_{-\lambda h(m/\lambda)\leq n\leq \lambda h(m/\lambda)}
\frac{1}{(m\tau+n)^2(m\tau+n+1)}\right]\\
&=\sum_{m\neq 0}\left[\sum_{n\in\mathbb{Z}}\frac{1}{(m\tau+n)^2(m\tau+n+1)}\right],
\end{align*}
where in the last equality we have appealed to absolute convergence to justify rearranging the series.
From the above we see that $$E(K,\tau)=G_2(K,\tau)-G_2(\tau)=\sum_{K} \left(\frac{1}{m\tau+n}-\frac{1}{m\tau+n+1}\right).$$ When we replace the latter sum with the integral expression from the lemma, the proof of Theorem~\ref{main-thm} is complete.
\qed

\section{$K$-summation of the Weierstrass $\wp$-function}

\label{sec:weierstrass}

Recall that the Weierstrass $\wp$-function is a function of two complex variables $\tau, z$ (with the dependence on $\tau$ usually suppressed in the notation) defined as
\begin{equation} \label{eq:wp-def}
\wp(z) :=\frac{1}{z^2}+\sum_{(m,n)\in \mathbb{Z}^2\setminus \{(0,0)\}}\left[\frac{1}{(z+n+m\tau)^2}-\frac{1}{(n+m\tau)^2}\right],
\end{equation}
for $\tau\in\mathbb{H}$ and $z\not\in \mathbb{Z}\tau+\mathbb{Z}$. 
The sum is absolutely convergent, but is only made so via the introduction of the ``normalization term''
$\frac{1}{(n+m\tau)^2}$. Indeed, the $\wp$-function is a fundamental object in the theory of doubly-periodic (or elliptic) functions, and the basic idea underlying its definition \eqref{eq:wp-def} is to try to construct a doubly-periodic function with the two periods $1,\tau$ by summing copies of a single term (for which the best choice turns out to be the meromorphic function $z^{-2}$, which has a pole of order~$2$ at the origin) translated over the lattice $\mathbb{Z}\tau+\mathbb{Z}$. This results in the series $\sum_{m,n} \frac{1}{(z+n+m\tau)^2}$, which however is only conditionally convergent. Subtracting $\frac{1}{(n+m\tau)^2}$ from each summand with $(m,n)\neq (0,0)$ turns the series into an absolutely convergent one, and, conveniently, still ends up producing a doubly-periodic function \cite[Ch.~1]{Apostol}.

Thus, we see that thinking about the definition of $\wp(z)$ and its motivation leads one to consider different orders of summation for the conditionally convergent infinite series $\sum_{m,n} \frac{1}{(z+n+m\tau)^2}$, in a way that is precisely analogous to the situation with $G_2$, and that --- it turns out --- will lead to exactly the same notion of ``residual function'' we already defined.

As with $G_2$, the two most obvious orders for summing the series are as iterated summations with respect to the summation indices $m,n$, in the two possible orders. The comparison between these two orders of summation is given in the following result (which is a more explicit version of a result stated implicitly as part of an exercise on p.~281 of the popular textbook \cite{Stein}), analogous to \eqref{eq:g2-order-summation}.

\begin{prop}
\label{prop:wp-twoorders}
For $z\not\in \mathbb{Z}\tau+\mathbb{Z},$ the following identity holds, where summations are in the indicated order, over all pairs $(m,n)\in\mathbb{Z}^2.$
\[
\sum_{m}\left[\sum_{n}\frac{1}{(z+n+m\tau)^2}\right]=\sum_{n}\left[\sum_{m}\frac{1}{(z+n+m\tau)^2}\right]+\frac{2\pi i}{\tau}.
\]
\end{prop}
\begin{proof}
We rewrite the sum defining $\wp(z)$ in the following way, excluding from the summations the terms corresponding to $m=n=0$.
\begin{align*}
\wp(z)&=\frac{1}{z^2}+\sum_{m}\left[\sum_{n} \left(\frac{1}{(z+n+m\tau)^2}-\frac{1}{(m\tau+n)^2}\right)\right]\\
&=\frac{1}{z^2}+\sum_{m}\left[\sum_{n} \frac{1}{(z+n+m\tau)^2}\right]-\sum_{m}\left[\sum_{n} \frac{1}{(m\tau+n)^2}\right]\\
&=\frac{1}{z^2}+\sum_{m}\left[\sum_{n} \frac{1}{(z+n+m\tau)^2}\right]-G_2(\tau).
\end{align*}
Meanwhile, we can sum the function $\wp(z)$ in the reverse order, by absolute convergence, obtaining 
\begin{align*}
\wp(z)&=\frac{1}{z^2}+\sum_{n}\left[\sum_{m} \left(\frac{1}{(z+n+m\tau)^2}-\frac{1}{(m\tau+n)^2}\right)\right]\\
&=\frac{1}{z^2}+\sum_{n}\left[\sum_{m} \frac{1}{(z+n+m\tau)^2}\right]-\sum_{n}\left[\sum_{m} \frac{1}{(m\tau+n)^2}\right]\\
&=\frac{1}{z^2}+\sum_{n}\left[\sum_{m} \frac{1}{(z+n+m\tau)^2}\right]-G_2(\tau)+\frac{2\pi i}{\tau},
\end{align*}
by \eqref{eq:g2-order-summation}. The proposition follows by comparing these two expressions for $G_2(\tau)$. 
\end{proof}

The same exercise in \cite{Stein} mentioned above also considers a summation order for the series $\sum_{m,n} \frac{1}{(z+n+m\tau)^2}$ similar to our notion of summation with respect to a shape $K$, in the special case where $K$ is a disk. We consider the more general shape summation for the series associated with the $\wp$-function through the following definition.

\begin{definition}
For $K\in\mathcal{K}$ and $\tau\in\mathbb{H}$, we denote by $\wp(K,z)$ the $K$-summation
\begin{equation*}
\wp(K,z):=\sum_{K}\frac{1}{(z+n+m\tau)^2} \qquad (z\not\in \mathbb{Z}\tau+\mathbb{Z}),
\end{equation*}
provided the limit defining the summation exists.
We refer to $\wp(K,z)$ as the \textbf{$K$-summation of the Weierstrass $\wp$-function.}
\end{definition}

The next result shows that $\wp(K,z)$ is closely related to the residual function $E(K,\tau)$ that we studied in the previous sections.

\begin{prop}
\label{main-cor}
For $K\in \mathcal{K}$, $\tau\in\mathbb{H}$ and $z \not\in \mathbb{Z}\tau+\mathbb{Z},$ $\wp(K,\tau)$ is defined and satisfies
\[\wp(K,z)=\wp(z)+G_2(\tau)+E(K,\tau).\]
\end{prop}
\begin{proof}
By absolute convergence of the sum defining $\wp(z)$, we can sum over the integer lattice points in any order without changing the function's value. Therefore,
\[
\wp(z)=\sum_{K}\left[\frac{1}{(z+n+m\tau)^2}-\frac{1}{(n+m\tau)^2}\right],
\]
where $a_{0,0}=\frac{1}{z^2}$ in \eqref{eq:shape-summation}.
Thus, 
\begin{align*}
\wp(z)&=\sum_{K}\frac{1}{(z+n+m\tau)^2}-\sum_{K}\frac{1}{(n+m\tau)^2}\\
&=\wp(K,z)-G_2(K,\tau)\\
&=\wp(K,z)-(G_2(\tau)+E(K,\tau)).
\end{align*}  
\end{proof}

\section{Examples}

The integral in the explicit formula \eqref{eq:residual} can sometimes be evaluated in closed form. Here are a few examples.

\begin{enumerate}

\item \textbf{Rectangle.}
As we saw in Section~\ref{sec:intro}, the residual function for the rectangle $[-c,c]\times[-1,1]$ with aspect ratio $c>0$ is
$E(K,\tau)=-\frac{4}{\tau}\tanh^{-1}(c \tau).$ 

\item \textbf{Disk.}
When $K$ is the disk of radius $1$ centered at the origin, we have $h_K(x)=\sqrt{1-x^2}$, $x\in [-1,1]$.
According to Theorem \ref{main-thm}, we have
\[
E(K,\tau)=4\int_{0}^1\frac{\sqrt{1-x^2}}{\tau^2x^2-(1-x^2)}\, dx=\frac{-2\pi i}{\tau+i}.
\]

\item \textbf{Diamond.}
For the last example we let $K$ be the diamond $\{ x+iy\,:\, |x|+|y|\le 1 \}$. Then $h_K(x)=1-|x|$, $x\in[-1,1]$. We have
\begin{align*}
E(K,\tau)&=4\int_{0}^1 \frac{1-x}{\tau^2x^2-(1-x)^2}\, dx=\frac{4\log(-i\tau)+2\pi i \tau}{1-\tau^2},
\end{align*}
where log denotes the principal branch of the logarithm.

\end{enumerate}

\section{Concluding Remarks}
\label{sec:concluding}

\begin{enumerate}

\item In this paper we constructed a large family of examples of natural rearrangements of conditionally convergent series. Our explicit formula \eqref{eq:residual} for the residual function $E(K,\tau)$ provides a general way to evaluate the discrepancy between any rearrangement in the family we considered and the ``default'' ordering of the series, thus generalizing the well-known relation \eqref{eq:g2-order-summation} and its equivalent version \eqref{eq:g2-order-summation-equivalent}. It is worth noting that the quasimodularity relation \eqref{eq:g2-order-summation-equivalent}, while being an example of ``bad behavior'' from the point of view of infinite series, is actually a ``good'' (i.e., useful, and important) property of $G_2$, forming the basis for the study of many of its properties as well as the properties of additional functions in complex analysis and the theory of modular forms that are constructed using $G_2$ as a building block. Some well-known applications of $G_2$ are to the study of the modular discriminant $\Delta$ \cite[pp.~20--21]{zagier123}; to proving the four square theorem that gives an explicit formula for the number of representations of an integer as a sum of four squares \cite[Ch.~10]{Stein}; and to constructing the ``magic function'' that played a crucial role in Viazovska's remarkable recent solution of the sphere packing problem in dimension 8 \cite{viazovska} (see also \cite{cohn}).

In view of the importance of \eqref{eq:g2-order-summation-equivalent}, it is interesting to wonder whether \eqref{eq:residual} might similarly provide fresh insight into some questions about modular forms that are of independent interest.

\item
Also related to \eqref{eq:g2-order-summation-equivalent} is the observation that for any $K\in\mathcal{K}$, if $K$ is symmetric about the line $y=x$, then the residual function $E(K,\tau)$ satisfies a similar functional equation, namely
\begin{equation}
\label{eq:residual-quasimodular}
E(K,\tau)=\tau^{-2}E(K,-1/\tau)-\frac{2\pi i}{\tau},
\end{equation}
which differs from the equation for $G_2(\tau)$ only in the sign of the $\frac{2\pi i}{\tau}$ term. This can be derived as follows. Given a shape $K$, we let $K^T$ be the shape obtained by reflecting $K$ about the line $x=y$. 

Replacing $\tau$ with $-1/\tau$ in the $K$-summation of $G_2(K,\tau)$, one obtains
\begin{align*}
G(K,-1/\tau)&=\tau^2G(K^T,\tau)\\
&=\tau^2(G_2(\tau)+E(K^T,\tau)).
\end{align*}
Meanwhile, the functional equation for $G_2(\tau)$ implies that
\begin{align*}
G(K,-1/\tau)&=G_2(-1/\tau)+E(K,-1/\tau)\\
&=\tau^2G_2(\tau)-2\pi i \tau +E(K,-1/\tau).
\end{align*}
Equating these two expressions for $G(K,-1/\tau)$ and subtracting the term $\tau^2G_2(\tau)$ from both sides gives
\begin{equation*}
E(K^T,\tau)=\tau^{-2}E(K,-1/\tau)-\frac{2\pi i}{\tau}.
\end{equation*}
If $K$ is symmetric about $y=x$, then $K^T=K$, so \eqref{eq:residual-quasimodular} holds. 

\item
We saw how a shape $K \in \mathcal{K}$ gives rise to residual functions, which are computed as a kind of integral transform of the associated bounding function $h_K$. It seems natural to try to reverse this correspondence and ask which holomorphic functions $f:\mathbb{H}\to\mathbb{C}$ occur as residual functions for shapes in~$\mathcal{K}$. We leave this as an open problem.

\item Finally, we note that one can consider summation with respect to shapes in greater levels of generality. In particular, one could expand the class of shapes $\mathcal{K}$ by relaxing the symmetry conditions, the condition of compactness, or both. We leave to the interested reader to work out the details of such generalizations.

\end{enumerate}

}


\bigskip
\begin{thebibliography}{999}

\bibitem{Apostol}
T.~Apostol. Modular Functions and Dirichlet Series in Number Theory, 2nd Ed.
Springer, 1994.

\bibitem{Bruin}
P.~Bruin, S.~Dahmen.
Modular Forms. Lecture Notes, 2016. Online resource, 
\texttt{http://www.few.vu.nl/\raisebox{-2pt}{\textasciitilde}sdn249/modularforms16/Notes.pdf}. Accessed October 23, 2018.

\bibitem{Diamond}
F.~Diamond, J.~Shurman.
A First Course in Modular Forms. Springer, 2016.

\bibitem{cohn}
H.~Cohn.
A conceptual breakthrough in sphere packing.
\textit{Not.~Amer.~Math.~Soc.} 64 (2017), 102--115.
  
\bibitem{eisenstein}
G.~Eisenstein. Mathematische Werke, Vol. 1. AMS Chelsea Publishing, 1975. 

\bibitem{grattan-guinness}
I.~Grattan-Guinness.
The Development of the Foundations of Mathematical Analysis from Euler to Riemann.
MIT Press, 1970.

\bibitem{rosenthal}
P.~Rosenthal.
The remarkable theorem of L\'evy and Steinitz.
\textit{Amer.~Math.~Monthly} 94 (1987), 342--351.

\bibitem{roy}
R.~Roy. Elliptic and Modular Functions: From Gauss to Dedekind to Hecke. Cambridge University Press, 2017. 

\bibitem{rudin}
W.~Rudin.
Principles of Mathematical Analysis, 3rd Ed.
McGraw-Hill, 1976.

\bibitem{Stein}
E.~M.~Stein, R.~Shakarchi.
Complex Analysis (Princeton Lectures in Analysis, No.~2). 
Princeton University Press, 2003.

\bibitem{viazovska}
M.~Viazovska.
The sphere packing problem in dimension~8.
\textit{Ann.~Math.} 185 (2017), 991--1015.

\bibitem{zagier123}
D.~Zagier. Elliptic modular functions and their applications. In: The 1-2-3 of Modular Forms, ed. K.~Ranestad, Springer, 2008, pp.~1--103.


\end{thebibliography}
\end{document}